\documentclass[12pt,reqno]{amsart}
\usepackage{amsfonts}
\usepackage{amssymb}
\usepackage{graphicx}
\usepackage{subfigure}
\usepackage{mathdots }
\usepackage{enumitem}
\usepackage{kbordermatrix}
\renewcommand{\kbldelim}{(}
\renewcommand{\kbrdelim}{)}

\def\B{{\mathbb B}} \def\C{{\mathbb C}}  \def\D{{\mathbb D}}   
   \def\R{{\mathbb R}} \def\SS{{\mathbb S}}  \def\T{{\mathbb T}}

     \def\cp{{\mathcal P}}

\newcommand{\al}{\alpha}   \newcommand{\de}{\delta} 
  \newcommand{\ve}{\varepsilon}  \newcommand{\ze}{\zeta}
\def\t{\theta}      \newcommand{\la}{\lambda}
  \newcommand{\vp}{\varphi}   \newcommand{\om}{\omega}  \newcommand{\Om}{\Omega}

\newtheorem{lem}{\sc Lemma}
\newtheorem{thm}{\sc Theorem}
  
\newtheorem{prob}{\sc Problem}  
\newtheorem{cor}{\sc Corollary}            

\newtheorem{thmother}{\sc Theorem} 
\newtheorem{lemother}[thmother]{\sc Lemma}

\theoremstyle{remark}  
\newtheorem{example}{Example} 
\newtheorem*{example*}{Example}

\begin{document}
\title{Schwarzian derivatives for pluriharmonic mappings}

\author[I. Efraimidis]{Iason Efraimidis}
\address{Department of Mathematics and Statistics, Texas Tech University, Box 41042, Lubbock, TX 79409, United States} 
\email{iason.efraimidis@ttu.edu}  

 \author[A. Ferrada-Salas]{\'Alvaro Ferrada-Salas}
 \address{Facultad de Matem\'aticas, Pontificia Universidad Cat\'olica de Chile, Santiago.} \email{alferrada@mat.uc.cl}

 \author[R. Hern\'andez]{Rodrigo Hern\'andez}
 \address{Facultad de Ingenier\'{\i}a y Ciencias, Universidad Adolfo Ib\'a\~nez, Av. Padre Hurtado 750, Vi\~na del Mar, Chile.} \email{rodrigo.hernandez@uai.cl}

 \author[R. Vargas]{Rodrigo Vargas}
 \address{Facultad de Matem\'aticas, Pontificia Universidad Cat\'olica de Chile, Santiago.} \email{rsvargas@mat.uc.cl}




\begin{abstract} 
A pre-Schwarzian and a Schwarzian derivative for locally univalent pluriharmonic mappings in $\C^n$ are introduced. Basic properties such as the chain rule, multiplicative invariance and affine invariance are proved for these operators.  It is shown that the pre-Schwarzian is stable only with respect to rotations of the identity. A characterization is given for the case when the pre-Schwarzian derivative is holomorphic. Furthermore, it is shown that if the Schwarzian derivative of a pluriharmonic mapping vanishes then the analytic part of this mapping is a M\"obius transformation. 

Some observations are made related to the dilatation of pluriharmonic mappings and to the dilatation of their affine transformations, revealing differences between the theories in the plane and in higher dimensions. An example is given that rules out the possibility for a shear construction theorem to hold in $\C^n$, for $n\geq2$.
\end{abstract}

\maketitle

\section{Preliminaries} 

\subsection{Introduction} The pre-Schwarzian and Schwarzian derivatives of a locally univalent holomorphic function $f$ in the plane are defined by 
$$
Pf \, = \, (\log f')' \, = \, \frac{f''}{f'} \qquad \text{and} \qquad Sf \, = \, (Pf)' - \tfrac{1}{2}(Pf)^2,  
$$
respectively. These operators appear frequently in geometric function theory and Teichm\"uller theory, where they are primarily used to prove criteria for univalence as well as criteria for quasiconformal and homeomorphic extension. They satisfy the invariance relations
$$
P(A\circ f) \, = \, Pf \qquad \text{and} \qquad S(T\circ f) \, = \, Sf 
$$
for linear (affine) maps $A(z)=az+b$, with $a\neq0$, and for M\"obius (linear fractional) transformations
$$
T(z) \, = \, \frac{az+b}{cz+d}, \qquad ad-bc\neq 0. 
$$ 
More generally, they satisfy the chain rules
$$
P(g \circ f) \, = \, (Pg \circ f) f' + Pf  \qquad \text{and} \qquad S(g\circ f) \, = \, (Sg \circ f) (f')^2 + Sf
$$
whenever the composition $g\circ f$ is well defined. We refer the reader to \cite{Leh} and \cite{Os98} for surveys of this classical theory. 

Both pre-Schwarzian and Schwarzian derivatives have been generalized and studied in the settings of harmonic mappings in the plane (see \cite{CDO03, HM15}) and of holomorphic mappings in $\C^n$ (see \cite{Pf74} and \cite{He06,MM96,Od74,OSt92}, for example). The main purpose of this article is to extend both operators to pluriharmonic mappings in $\C^n$. We introduce a pre-Schwarzian derivative that simultaneously generalizes the pre-Schwarzian derivatives of the two more restricted settings given by Mart\'in and the third author \cite{HM15} and Pfaltzgraff \cite{Pf74}. Moreover, we introduce a Schwarzian derivative that generalizes the Schwarzian operator for holomorphic mappings in $\C^n$ defined by the third author \cite{He06} and based on the work of Oda \cite{Od74}. 

We will be using the notations $Pf$ and $Sf$ when we know that the mapping $f$ is holomorphic either in one or in several complex variables and the notations $P_f$ and $S_f$, with the mapping $f$ as a subscript, in the more general settings of harmonic mappings in the plane or pluriharmonic mappings in $\C^n$. 

A pluriharmonic mapping $f$ in a simply connected domain $\Om\subset\C^n$ is a mapping $f=h+\overline{g}$, where $h$ and $g$ are holomorphic mappings in $\Om$ with values in $\C^n$. We assume throughout the article that $h$ is locally biholomorphic or, equivalently, that ${\rm det}Dh(z)\neq0$ for all $z\in\Om$, where $Dh$ is the $n\times n$ matrix $\big(\tfrac{\partial h_i}{\partial z_j}\big)_{i,j=1,\ldots n}$. Then the dilatation of $f$ is defined as $\om = Dg Dh^{-1}$. The real Jacobian of $f$ is given by 
\begin{equation*}
J_f(z) \, = \, {\rm det} \left(
\begin{array}{cc}
Dh(z)  & \overline{Dg(z)} \\
Dg(z)  & \overline{Dh(z)}
\end{array} \right), \qquad z\in\Om, 
\end{equation*}
or equivalently, by 
\begin{equation}\label{Jacob-pluri}
J_f(z) \, = \, |{\rm det}Dh(z)|^2 {\rm det} \, \big( I_n - \om(z) \overline{\om(z)} \,\big), \qquad z\in\Om, 
\end{equation}
where $I_n$ denotes the identity matrix of size $n$. Recent developments on pluriharmonic mappings include two-point distortion theorems \cite{DHK11,HHK13}, a Schwarz Lemma and Landau and Bloch theorems \cite{CG11,HK15}, as well as theorems on stable univalent mappings and univalence criteria through the use of linearly connected domains \cite{CHHK14}, among others. A sufficient condition for $f$ to be sense-preserving, \emph{i.e.} for $J_f>0$, was given in Theorem 5 of \cite{DHK11}. The norm notation hereafter refers to the standard operator norm (see \S~\ref{sub-sect-1.2}).

\begin{thmother}[\cite{DHK11}] \label{DHK-om-sense-preserv}
If $f=h+\overline{g}$ is a pluriharmonic mapping in $\Om$ for which $\|\om(z)\|<1$ for all $z\in\Om$ then $f$ is sense-preserving in $\Om$. 
\end{thmother}

This was formulated in \cite{DHK11} for the case when $\Om$ is the unit ball in $\C^n$, but the proof does not depend on the domain of definition. In Section~\ref{sect-4} we make the simple observation that the converse of this fails, giving an example of a mapping $f$ for which $J_f>0$, but whose dilatation has arbitrarily large norm.  

Let $A$ be a complex matrix that satisfies $\|A\|<1$ and consider the affine mapping of the form $z\mapsto z + A \, \overline{z}$. Post-composing with a pluriharmonic mapping $f=h+\overline{g}$ we obtain a new pluriharmonic mapping 
$$
F  \, = \, f+A \,\overline{f}  \, = \,  H+\overline{G}
$$
with
$$
H  \, = \,  h+A \, g \qquad \text{and} \qquad G  \, = \, g + \overline{A} \, h.
$$
Assume now that $\|\om(z)\|<1$ for all $z\in\Om$. Then clearly the mapping $I_n+A\,\om(z)$ is invertible and we may compute the dilatation of $F$ as 
\begin{align} \label{om_F}
\om_F \, = &\, DG \, DH^{-1}  \nonumber \\
 = & \, \big( (\om + \overline{A}) Dh \big) \, \big( (I_n+A\,\om ) Dh \big)^{-1}  \nonumber  \\ 
 = &  \, (\om + \overline{A})(I_n+A\,\om )^{-1}. 
\end{align}
Observe that in the plane, the above takes the familiar form of a disk-automorphism and its modulus is therefore bounded by 1. It is then natural to ask if the same holds in dimension $n\geq 2$, that is, if $\|\om\|<1$ implies $\|\om_F\|<1$. However, the answer to this is negative and therefore Theorem~\ref{DHK-om-sense-preserv} can not be applied to $F$. Nevertheless, the assumption that $\|\om\| <1$ is sufficient for $F$ to be sense-preserving. We have that 
\begin{enumerate}[itemsep=5pt]
\item[(i)] $\|\om\|<1$ does not imply $\|\om_F\|<1$

\item[(ii)] $\|\om\|<1$ implies ${\rm det}( I_n - \om_F \overline{\om_F})> 0$

\item[(iii)] ${\rm det}( I_n - \om \overline{\om})> 0$ does not imply ${\rm det}( I_n - \om_F \overline{\om_F})> 0$. 
\end{enumerate}
Proposition (ii) is proved in Theorem~\ref{factor}, while counterexamples proving (i) and (iii) are given in Section~\ref{sect-4}, in Example~\ref{counter-omega} and Example~\ref{counter-det}, respectively. We perceive from this list that, in contrast to the planar case where we may apply an arbitrary number of affine transformations and each time obtain a sense-preserving harmonic mapping, it seems that in higher dimensions we are only allowed to apply one. It is therefore interesting to ask the following.
\begin{prob} \label{prob-1} 
Does there exist a condition on the dilatation $\om$ which implies \linebreak ${\rm det}( I_n - \om \overline{\om}) > 0$ and ${\rm det}(I_n+A\,\om)\neq0$ for every matrix $A$ with $\|A\|<1$ and, moreover, is preserved under affine transformations? 
\end{prob}

In Section~\ref{sect-2} we consider the class
\begin{equation}\label{class-loc-univ-pluri}
\cp(\Om) \, = \, \{ f:\Om\to\C^n \, | \, f\, \text{pluriharmonic}, J_h\neq0 \; \text{and}\; J_f\neq0 \; \text{throughout} \;\Om \} 
\end{equation}
and define the pre-Schwarzian derivative of a mapping $f=h+\overline{g}$ in $\cp(\Om)$ with dilatation $\om$ as the bilinear mapping 
\begin{equation} \label{def-pre-S-form}
P_f \langle \cdot,\cdot \rangle \, =\, Ph\langle \cdot,\cdot \rangle  - Dh^{-1} \big(I_n - \overline{\om}\om \big)^{-1} \, \overline{\om} \, D\om\langle \cdot,Dh \, \cdot \rangle. 
\end{equation}
We prove the chain rule for $P_f$ and show that $P_f$ is invariant under multiplication by an invertible matrix and under affine transformations. We then further justify our definition of $P_f$ via best affine approximation. We characterize the case when $P_f$ is holomorphic and deduce from it sufficient conditions for $f$ to be univalent. We show that $P_f$ is stable only with respect to rotations of the identity. We ask the following question for the unit ball $\B^n$ in $\C^n$.
\begin{prob}
Does there exist a constant $c\leq 1$ such that for every pluriharmonic mapping $f$ in $\B^n$ the condition 
$$
(1-|z|^2)\|P_f(z)\|\leq c, \qquad z\in \B^n,
$$
is sufficient for $f$ to be univalent? 
\end{prob}
For $f$ holomorphic Pfaltzgraff \cite{Pf74} answered this in the affirmative with $c=1$, which is sharp. 

In Section~\ref{sect-3} we define the Schwarzian derivative $S_f$ of a pluriharmonic mapping $f$ and deduce its basic properties from the corresponding properties of $P_f$. We prove that if $S_f$ vanishes then the holomorphic part of $f$ is a M\"obius transformation.

\subsection{Preliminaries in several complex variables} \label{sub-sect-1.2} Let $\C^n$ be the space of points $z=(z_1,\ldots,z_n)$, where each $z_i\in\C$. It is endowed with the inner (dot) product $z\cdot w =\sum_{i=1}^n z_i \overline{w_i}$ and the norm $|z|= (z\cdot z)^{1/2}$. 

We denote by $\mathcal L^k(\C^n)$ the space of continuous $k$-linear operators from $\C^n$ into $\C^n$. For $T \in \mathcal L^k(\C^n)$ we write $T\langle\cdot,\ldots,\cdot\rangle$ to denote its placeholders. The standard operator norm in $\mathcal L^k(\C^n)$ is given by 
$$
\|T\| \, = \, \max_{u_1,\ldots,u_k \in\C^n\backslash\{0\}} | T\langle \tfrac{u_1}{|u_1|},\ldots,\tfrac{u_k}{|u_k|}\rangle |. 
$$
When $k=1$ we simply write $\mathcal L(\C^n)$ for the space of linear maps and also write $Tu$ instead of $T\langle u\rangle$ for any linear map $T$. 

Let $\Om$ be a domain in $\C^n$ and $f$ a mapping in $\Om$ with values in $\mathcal L(\C^n)$. If $f$ is $k$-times (Fr\'echet) differentiable with respect to $z\in\Om$ then its $k$-th derivative, denoted by $D^kf(z)$, is a symmetric mapping in $\mathcal L^{k+1}(\C^n)$, meaning that the value $D^kf(z)\langle u_1,\ldots,u_{k+1}\rangle$ remains unchanged after any permutation of the entries $u_1,\ldots,u_{k+1}$ (see Theorem 14.6 in \cite{Mu}). In the sequel we will be needing the \emph{product rule} for the derivative of the product $fg$ of two differentiable mappings $f, g: \Omega\to\mathcal L(\C^n)$, \emph{i.e.}, the composition of the linear mappings $f(z)$ and $g(z)$. Using the definition of differentiability one can easily prove that 
\begin{equation} \label{prod-rule}
D(fg)(z)\langle\cdot,\cdot\rangle  \, = \, Df(z)\langle\cdot,g(z) \,\cdot\rangle +f(z) \,Dg(z)\langle\cdot,\cdot\rangle, \quad z\in\Omega. 
\end{equation}
From this we can see that if $f(z)$ is invertible for every $z\in\Om$ then the derivative of $g(z)=f(z)^{-1}$  is given by
\begin{equation} \label{deriv-inverse}
Dg(z)\langle \cdot,\cdot\rangle\, = \, -g(z) Df(z)\langle \cdot,g(z)\,\cdot\rangle, \qquad z\in\Om.
\end{equation}
Finally, if $f: \Omega\to\mathcal L(\C^n)$ is holomorphic then Taylor's formula centered at some $\al\in\Om$ reads 
$$
f(z) \, = \, \sum_{k=0}^\infty \frac{1}{k!}D^kf(\al)(z-\al)^k, \qquad |z-\al|<\de(\al), 
$$
where $\delta(\al)$ denotes the distance from $\al$ to the boundary of $\Om$; see Theorem 7.13 in \cite{Mu}. Here the notation $D^kf(\al)w^k$ should be understood as $D^kf(\al)\langle w,\ldots,w,\cdot\rangle$, with the point $w$ repeated $k$ times and one placeholder left without being evaluated.

\subsection{Pluriharmonic mappings} A function $u$ of class $C^2$ defined in a domain $\Om$ of $\C^n$ is called pluriharmonic if its restriction to every complex line is harmonic, that is, if for every fixed $z\in\Om$ and direction $\t\in\C^n, |\t|=1$, the function $u(z+\ze\t)$ is harmonic in $\{\ze\in\C  :  z+\ze\t\in\Om \}$. Equivalently, $u$ is pluriharmonic if 
$$
\frac{\partial^2 u}{\partial z_j \partial\overline{z_k}} \, = \, 0, \qquad \text{for all} \quad  j,k=1,2,\ldots n.  
$$
Interest in these functions stems from the fact that in a simply connected domain the class of real-valued pluriharmonic functions coincides with the class consisting of the real part of holomorphic functions. In contrast to the planar case, where every real-valued harmonic function is the real part of a holomorphic function, in dimension $n\geq2$ the pluriharmonic functions form a proper subclass of harmonic functions. See \cite{Kr} or \cite{Ru80} for basic facts about pluriharmonic functions. A \emph{pluriharmonic mapping} $f:\Om\to\C^n$ is a mapping all of whose coordinates are complex-valued pluriharmonic functions. 

If $\Om$ is simply connected then every pluriharmonic mapping $f:\Om\to\C^n$ can be written as $f=h+\overline{g}$, with $h$ and $g$ holomorphic in $\Om$. To see this we may assume that $f$ has values in $\C$ by considering its coordinates.  We write $f=u+iv$. Since $\Om$ is simply connected we can analytically complete the pluriharmonic functions $u$ and $v$ throughout $\Om$ (see Theorem 4.4.9 in \cite{Ru80} and the comments thereafter), that is, we can find holomorphic mappings $a$ and $b$ in $\Om$ for which $u={\rm Re}\,a$ and $v={\rm Re}\,b$. Setting $h=\frac{1}{2}(a+ib)$ and $g=\frac{1}{2}(a-ib)$ we easily verify that $f=h+\overline{g}$. This representation is unique up to an additive constant. We say that it is the \emph{canonical} representation of $f$ if $g(z_0)=0$ for some fixed point $z_0$ in $\Om$.

\subsection{Affine transformations} Let $f=h+\overline{g}$ be a pluriharmonic mapping with dilatation $\om$ and let $A$ be a complex matrix that satisfies $\|A\|<1$. Then the affine transformation $F = f+A \,\overline{f}$ has dilatation $\om_F \, =  \, (\om + \overline{A})(I_n+A\,\om )^{-1}$, as seen in \eqref{om_F}. The following theorem shows that the assumption $\|\om\| <1$ is sufficient for $F$ to be sense-preserving. Its statement is purely about matrices though we maintain the notation of pluriharmonic mappings in order to keep it in context. 

\begin{thm} \label{factor}
If the complex matrices $A$ and $\om$ satisfy $\|A\|<1$ and $\|\om\| <1$ then the matrix
$$
\om_F \, =\, (\om + \overline{A}) (I_n +A\om)^{-1}
$$
satisfies 
$$
I_n - \om_F \overline{\om_F}  \, = \, (I_n - \overline{A} A ) (I_n +A\om)^{-1} ( I_n - \om \overline{\om} )  ( I_n +\overline{A} \overline{\om} )^{-1}
$$
and ${\rm det}( I_n - \om_F \overline{\om_F})>0$.  
\end{thm}
\begin{proof}
Let $B=I_n - \om_F \overline{\om_F}$ and compute 
\begin{align*}
C \, & = \, B ( I_n +\overline{A} \overline{\om} ) \\ 
& = \,   I_n +\overline{A} \overline{\om} - (\om + \overline{A}) (I_n +A\om)^{-1}(  \overline{\om} +A). 
\end{align*}
We add and subtract the term $\overline{A}A$ to get 
\begin{align*}
C \, & = \,  I_n - \overline{A}A + \big[\, \overline{A} - (\om + \overline{A}) (I_n +A\om)^{-1} \big] ( \overline{\om} +A)  \\ 
& = \,  I_n - \overline{A}A + \big[ \, \overline{A}(I_n +A\om) - (\om + \overline{A}) \big] (I_n +A\om)^{-1} ( \overline{\om} +A)  \\ 
& = \,  I_n - \overline{A}A - ( I_n - \overline{A}A ) \, \om \, (I_n +A\om)^{-1} ( \overline{\om} +A). 
\end{align*}
Setting 
\begin{align*}
D \, & = \, ( I_n - \overline{A}A )^{-1} C \\
& = \, I_n -  \om \, (I_n +A\om)^{-1} ( \overline{\om} +A)  
\end{align*}
and noting that 
$$
\om \, (I_n +A\om)^{-1} \, = \, (I_n + \om A )^{-1}\om 
$$
we compute 
\begin{align*}
D \, & =  \, I_n -  (I_n +\om A)^{-1}\om \, ( \overline{\om} +A)  \\
 & =  \, (I_n +\om A)^{-1}\big[ I_n +\om A- \om \, ( \overline{\om} +A)  \big]  \\
  & =  \, (I_n +\om A)^{-1} (I_n -\om  \overline{\om}). 
\end{align*}
We complete the factorization of $B$ by writing 
\begin{align*}
B \, & = \, C \, ( I_n +\overline{A} \overline{\om} )^{-1} \\
 & = \,  ( I_n - \overline{A}A ) \, D\, ( I_n +\overline{A} \overline{\om} )^{-1}\\
  & = \,  ( I_n - \overline{A}A ) (I_n +\om A)^{-1} (I_n -\om  \overline{\om}) ( I_n +\overline{A} \overline{\om} )^{-1}. 
\end{align*}

In order to show that ${\rm det}( I_n - \om_F \overline{\om_F})> 0$ we may replicate an idea from the proof of Theorem 5 in \cite{DHK11}: Note first that the assumptions $\|A\|<1$ and $\|\om\| <1$ and the above factorization readily imply that ${\rm det}( I_n - \om_F \overline{\om_F})\neq 0$. Consider parameters $s$ and $t$ in $[0,1]$ and see that all previous arguments apply to the matrices $sA$ and $t\om$ so that the corresponding determinant is also different from zero. Since it is equal to 1 for $s=t=0$, continuity with respect to $s$ and $t$ shows that it is positive for $s=t=1$. 
\end{proof}

\section{Pre-Schwarzian derivative} \label{sect-2}

\subsection{Previous definitions of the pre-Schwarzian} For a locally biholomorphic mapping $f:\Om\to \C^n$, where $\Om$ is some domain in $\C^n$, we adopt the definition of the pre-Schwarzian derivative as a bilinear mapping given by 
\begin{equation}\label{Pfhol}
Pf(z)\langle\cdot,\cdot\rangle=Df(z)^{-1} D^2f(z)\langle\cdot,\cdot\rangle, \qquad z\in\Om. 
\end{equation}
This was introduced by Pfaltzgraff in \cite{Pf74}, who mostly considered the linear mapping $Pf(z)\langle z,\cdot\rangle$.

On the other hand, for a locally univalent harmonic mapping $f=h+\overline{g}$ on a planar domain $\Om$, with dilatation $\omega=g'/h':\Omega\to\mathbb{D}$, the definition 
\begin{equation} \label{Hern-Mart}
P_f(z) \, = \, (\log J_f(z))_z \, = \, Ph(z) - \frac{\overline{\om(z)}\om'(z)}{1-|\om(z)|^2}, \qquad z\in \Om,  
\end{equation}
was introduced in \cite{HM15} by Mart\'in and the third author. Here, in accordance with \eqref{Jacob-pluri}, the Jacobian is $J_f=(1-|\om|^2)|h'|^2$. Note that the anti-analytic term in the expression 
$$
\log J_f  \, = \,   \log (1-|\om|^2) +\log h' + \log\overline{h'} 
$$
plays no role in the differentiation and, hence, we may use the expression 
\begin{equation}\label{Uop}
U \, = \, (1-|\om|^2) h'
\end{equation}
and obtain the exact same pre-Schwarzian by setting $P_f=(\log U)_z$. Exploiting this observation we propose a definition of the pre-Schwarzian derivative for pluriharmonic mappings in several complex variables.  

\subsection{Definition of $P_f$ for pluriharmonic mappings} Let $\Omega$ be a simply connected domain in $\C^n$ and $f\in\cp(\Om)$, the class of pluriharmonic mappings given in~\eqref{class-loc-univ-pluri}. We consider 
$$ 
U(z) \, = \, \big(I_n - \overline{\om(z)}\om(z) \big) Dh(z), \qquad z\in\Om,
$$ 
and define the \emph{pre-Schwarzian} derivative of $f$ as the mapping $P_f: \Omega\to\mathcal L^2(\C^n)$ given by
\begin{equation} \label{def-pre-S}
P_f(z)\langle \cdot,\cdot\rangle \, = \, U(z)^{-1} \, DU(z)\langle \cdot,\cdot\rangle, \qquad z\in \Om. 
\end{equation}
We apply the product rule \eqref{prod-rule} and, suppressing the variable $z\in\Om$, compute
$$
DU\langle \cdot,\cdot \rangle  \, = \, -\overline{\om} \,  D\om\langle \cdot,Dh \, \cdot \rangle +(I_n-\overline{\om}\om)D^2h\langle \cdot,\cdot \rangle. 
$$
Therefore, we obtain from \eqref{def-pre-S} that
$$
P_f \langle \cdot,\cdot \rangle \, =  \,  Dh^{-1}D^2h\langle \cdot,\cdot \rangle  - Dh^{-1} \big(I_n - \overline{\om}\om \big)^{-1} \overline{\om}D\om\langle \cdot,Dh \, \cdot \rangle,
$$
which is equivalent to the expression \eqref{def-pre-S-form}. This operator is evidently a generalization of both the pre-Schwarzian for holomorphic mappings in several complex variables \eqref{Pfhol} and the pre-Schwarzian for planar harmonic mappings \eqref{Hern-Mart}. Further justification for this definition will be provided in \S~\ref{sub-sect-2.4} with the process of best affine approximation.

\subsection{Chain rule, multiplicative invariance and  affine invariance} We first establish the chain rule for the composition of a pluriharmonic mapping with a locally biholomorphic mapping. Hereafter we tacitly assume in all statements that the necessary assumptions on the mappings involved hold so that their pre-Schwarzian derivatives are well defined. 

\begin{thm}[Chain rule for $P_f$] \label{chain-rule-pre-Schw}
Let $f=h+\overline{g}$ be a pluriharmonic mapping and $\varphi$ be a locally biholomorphic mapping such that the composition $f\circ\varphi$ is well defined. Then 
$$
P_{f\circ\varphi}(z)\langle \cdot,\cdot \rangle \, = \, D\varphi(z)^{-1} P_f\big(\varphi(z)\big) \langle D\varphi(z)\, \cdot, D\varphi(z) \, \cdot \rangle +P\varphi(z) \langle\cdot,\cdot\rangle 
$$
for all $z$ in the domain of definition of $\varphi$. 
\end{thm}
\begin{proof}
We set $F=f\circ\varphi$ and write $F=H+\overline{G}$, where $H=h\circ\varphi$ and $G=g\circ \varphi$. We also note that $\omega_F=\omega\circ\varphi$. In order to use formula \eqref{def-pre-S-form} for the mapping $F$ we compute
$$
DH(z) \, = \, Dh\big(\varphi(z)\big) \, D\varphi(z) 
$$
and
$$
D^2H(z)\langle \cdot,\cdot \rangle \, = \, D^2h\big(\varphi(z)\big)\langle D\varphi(z) \cdot \, , D\varphi(z) \cdot \rangle + Dh\big(\varphi(z)\big) D^2\varphi(z)\langle \cdot,\cdot \rangle. 
$$
Therefore
\begin{align*}
PH(z)\langle \cdot,\cdot \rangle \, = & \, DH(z)^{-1} \, D^2H(z)\langle \cdot,\cdot \rangle  \\
= & \, D\varphi(z)^{-1}Ph\big(\varphi(z)\big) \langle D\varphi(z)\,\cdot, D\varphi(z)\,\cdot \rangle +P{\varphi}(z)\langle \cdot,\cdot \rangle. 
\end{align*}
Finally, we compute
$$
D\omega_F(z)\langle \cdot,\cdot \rangle  \, = \, D\omega\big(\varphi(z)\big) \langle D\varphi(z)\cdot \, ,\cdot \rangle 
$$
and arrive at the desired conclusion after a substitution in \eqref{def-pre-S-form}. 
\end{proof}

The following is a generalization of Lemma 1 in \cite{HM15}. It shows that the pre-Schwarzian of a pluriharmonic mapping is equal to the pre-Schwarzian of a specific holomorphic mapping which, however, depends on the point of evaluation. 
\begin{lem}\label{pre-S-lem}
Let $f=h+\overline{g}$ be a pluriharmonic mapping in $\Om$ having dilatation $\om$. Then  
$$
P_{f}(z_0) \, = \, P\big(h-\overline{\omega(z_0)}g\big)(z_0) 
$$
for any fixed $z_0$ in $\Om$. 
\end{lem}
\begin{proof}
Let $F=h+A\,g$, where $A$ is a matrix with $\|A\|<1$. We compute
$$
DF \, = \, (I_n+A\,\om ) Dh 
$$
and use the product rule \eqref{prod-rule} in order to find that
$$
D^2F \langle \cdot \,,\cdot \rangle \, = \, A\,D\om\langle \cdot \,, Dh \cdot \rangle+ (I_n+A\,\om ) D^2h\langle \cdot \,,\cdot \rangle .
$$
Therefore, in view of definition \eqref{Pfhol} we have that
$$
PF \langle \cdot \,,\cdot \rangle \, = \, Ph \langle \cdot \,,\cdot \rangle +Dh^{-1} (I_n+A\,\om )^{-1}A\,D\om\langle \cdot \,,Dh\cdot \rangle. 
$$ 
We now get the desired result by taking $A=-\overline{\omega(z_0)}$ and evaluating at $z_0$. 
\end{proof}

Note that the pre-Schwarzian \eqref{Pfhol} of a holomorphic mapping $f$ remains unchanged if $f$ is multiplied by an invertible $n\times n$ complex matrix $A$. We readily see this by writing $F=Af$ and computing 
\begin{equation} \label{holo-mult-inv}
PF \, = \,  DF^{-1} D^2F \, = \, (A\,Df)^{-1} A\, D^2f \,  = \, (Df)^{-1} D^2f\, = \, Pf. 
\end{equation}
We now generalize this to pluriharmonic mappings. 

\begin{thm}[Multiplicative invariance for $P_f$] \label{harm-mult-inv}
If $f$ is a pluriharmonic mapping and $A$ is an invertible matrix then 
$$
P_{Af} = \, P_f.
$$
\end{thm}
\begin{proof}  
Let $f=h+\overline{g}$ and set $F=Af =H+\overline{G}$, where $H=Ah$ and $G=\overline{A} g$. We compute the dilatation of $F$ as
$$
\om_F \, = \, DG \, DH^{-1}  \, = \, (\overline{A} Dg) (A Dh)^{-1} \, = \, \overline{A}\, \om A^{-1}. 
$$
From Lemma~\ref{pre-S-lem} and equation \eqref{holo-mult-inv} we get that 
\begin{align*}
P_F(z_0) \, = & \, P\big(H-\overline{\om_F(z_0)}G \big)(z_0) \\ 
 = & \, P\big[A \big(h-\overline{\omega(z_0)}g\big)\big](z_0) \\
 = & \, P\big(h-\overline{\omega(z_0)}g\big)(z_0) \\
 = & \, P_f(z_0), 
\end{align*}
which completes the proof since $z_0$ was arbitrary. 
\end{proof}

We now show that the pre-Schwarzian derivative \eqref{def-pre-S-form} is invariant under composition with affine transformations. 

\begin{thm}[Affine invariance for $P_f$] \label{harm-afin-inv}
If $f$ is a pluriharmonic mapping then 
$$
P_{f+A \, \overline{f}} \, = \, P_{f}
$$
for every matrix $A$ with $\|A\|<1$.
\end{thm}
\begin{proof}
We write $f=h+\overline{g}$ and set $F=f+A \,\overline{f}$, for which we have the decomposition $F=H+\overline{G}$, where $H=h+A\,g$ and $G=g + \overline{A}\,h$. According to \eqref{om_F}, the dilatation of $F$ is given by $\omega_F \, = \, (\omega + \overline{A})(I_n+A\omega)^{-1}$. In order to solve for $\om$ we write
$$
(I_n-\omega_FA) \omega \, = \, \omega_F -\overline{A}. 
$$
We claim that $I_n-\omega_FA$ is an invertible matrix. To see this we first note the elementary equality 
$$
(I_n+A\om )^{-1}A \, = \, A (I_n+\om A)^{-1} 
$$
and then compute 
\begin{align*}
I_n-\omega_F A \,  = & \, I_n- (\omega + \overline{A})(I_n+A\omega)^{-1}A \\ 
 = & \, I_n- (\omega + \overline{A}) \, A \, (I_n + \omega A)^{-1}  \\ 
= & \, \big[ I_n + \om A - (\omega + \overline{A}) A \,\big] (I_n+ \om A)^{-1}  \\
= & \, ( I_n  - \overline{A} A ) (I_n+ \om A)^{-1},  
\end{align*}
which proves our claim. Therefore we have
\begin{equation} \label{om-om_F}
\omega \, = \, (I_n-\omega_FA)^{-1}(\omega_F-\overline{A} \, ).
\end{equation}
We fix a point $z_0$ and make the following computation: 
\begin{align*}
H-\overline{\omega_F(z_0)}G \, = \, & h+A\,g -\overline{\omega_F(z_0)}(g+\overline{A}\,h)\\
= \, & \big( I_n - \overline{\om_F(z_0) A} \,\big)h +\big( A - \overline{\om_F(z_0)} \, \big) g \\ 
= \, & \big( I_n - \overline{\om_F(z_0) A} \, \big) \Big( h +\big(I_n-\overline{\om_F(z_0) A} \, \big)^{-1} \big( A - \overline{\om_F(z_0)} \, \big) g\Big) \\
= \, & \big( I_n - \overline{\om_F(z_0) A} \, \big) \big(h-\overline{\omega(z_0)}g \big), 
\end{align*}
where we made use of \eqref{om-om_F} at the last step. Since $I_n-\overline{\om_F(z_0) A}$ is a constant and invertible matrix we have, in view of \eqref{holo-mult-inv}, that 
$$
P\big(H-\overline{\omega_F(z_0)}G\big) \, = \, P\big(h-\overline{\omega(z_0)}g\big). 
$$
Now, with two applications of Lemma~\ref{pre-S-lem} we get that 
$$
P_F(z_0)\, = \, P\big(H-\overline{\omega_F(z_0)}G\big)(z_0) \, = \, P\big(h-\overline{\omega(z_0)}g\big)(z_0) \, = \,P_f(z_0).
$$
Since $z_0$ was arbitrary the proof is complete. 
\end{proof}

\subsection{Best affine approximation} \label{sub-sect-2.4} Our definition is primarily justified by the fact that it coincides with the second (analytic) derivative of the affine deviation of~$f$. To see this let $T$ be the best affine approximation of $f$ at the origin, \textit{i.e.}, the affine map that agrees with $f$ at its value and first analytic and anti-analytic derivatives. Clearly, we may assume that $f(0)=0$ and, in view of Theorem~\ref{harm-mult-inv}, that $Dh(0)=I_n$. Hence 
$$
T(z)=z+\overline{\om(0)} \, \overline{z}. 
$$
The \textit{affine deviation} of $f$ is then defined by $F=T^{-1}\circ f$. We compute $F=  C \, \big( f -\overline{\om(0)} \, \overline{f} \big)$, where $C=\big(I_n - \overline{\om(0)} \, \om(0) \big)^{-1}$. Writing $F=H+\overline{G}$ we find that 
$$
H\, = \, C \, \big( h -\overline{\om(0)} \, g \big) \qquad  \text{and} \qquad   G\, = \, \overline{C} \, \big( g - \om(0) h \big),  
$$
and, therefore, that
$$
DH\, = \, C \, \big( I_n -\overline{\om(0)} \, \om \big)Dh  \qquad  \text{and} \qquad   DG\, = \, \overline{C} \, \big( \om - \om(0)  \big) Dh.   
$$
Since $DH(0)=I_n$ and $DG(0) = 0$ we get from formulas \eqref{def-pre-S-form} and \eqref{Pfhol} that 
$$
P_F(0) = PH(0) = D^2H(0). 
$$
Noting that $P_f=P_F$ by the affine invariance shown in Theorem~\ref{harm-afin-inv} we deduce that $P_f(0)= D^2H(0)$. 

For an arbitrary point $a\in\Om$ we consider the translation $f_a(z) = f(z+a)$ and argue as above in order to compute the analytic part of the affine deviation $F_a$ of $f_a$ at the origin as 
$$
H_a(z) \, = \,Dh(a)^{-1} \big(I_n - \overline{\om(a)} \, \om(a) \big)^{-1} \big( h(z+a) - \overline{\om(a)} \, g(z+a) \big), 
$$
where $z\in\Om-a$. Hence, we find that
$$
P_f(a) \, = \, P_{f_a}(0) \, = \, P_{F_a}(0)  \, = \, PH_a(0) \, = \, D^2H_a(0) 
$$
and conclude that our definition of the pre-Schwarzian derivative is equal to the second (analytic) derivative of the affine deviation of a pluriharmonic mapping and, thus, being in line with ideas of Cartan \cite{Car} on the Schwarzian derivative, further exploited by Tamanoi \cite{Tam}, it constitutes a natural definition.

\subsection{Holomorphic or vanishing pre-Schwarzian} For a holomorphic map $f$ it is elementary to see from definition \eqref{Pfhol} that if $Pf= 0$ in some open set then $f$ is linear, that is, $f(z)=Az+b$, for some invertible matrix $A$ and some $b\in\C^n$. 
 
In the plane on the other hand, in view of definition \eqref{Hern-Mart}, if the pre-Schwarzian of a harmonic mapping $f=h+\overline{g}$ is holomorphic then a differentiation with respect to $\overline{z}$ gives  
$$
 \frac{\overline{\om' }\om' }{(1-|\om |^2)^2}  \, = \, 0, 
$$
from which it follows that $\om$ is constant. 

However, it is relatively easy to produce examples of pluriharmonic mappings in $\C^2$ with vanishing pre-Schwarzian~\eqref{def-pre-S-form} but non-constant dilatation. 
 
\begin{example} 
Let $\phi$ be a holomorphic function in some domain $\Om_1$ in $\C$ and consider the mapping
$$
f(z,w) \, = \, \big( z , w+\overline{\phi(z)}\big)
$$
which is pluriharmonic in $\Om_1\times\C$ and also, clearly, univalent. A representation for $f$ is given by
$$
h(z,w) \, = \,(z,w) \qquad \text{and} \qquad g(z,w) \, = \, \big( 0 , \phi(z)\big), 
$$
and its dilatation is 
$$ 
\om(z,w)\, = \, Dg(z,w) \, = \,
\begin{pmatrix}
0         & 0\\
\phi'(z) & 0
\end{pmatrix}, 
$$
which is non-constant if we choose $\phi$ to be non-linear. For the pre-Schwarzian of $f$ we easily compute 
$$
U \, = \,  \big(I_2 - \overline{\om}\,\om\big) Dh \, = \, I_2 
$$
and therefore $P_f =0$ in view of~\eqref{def-pre-S}. 
\end{example}

Our main theorem here characterizes the case when the pre-Schwarzian derivative of a pluriharmonic mapping in the general class \eqref{class-loc-univ-pluri} is holomorphic. 

\begin{thm} \label{vanish-pre-S}
Let $f\in\cp(\Om)$ be such that $\om(z_0)=0$ for some $z_0\in \Omega$. Then $P_f$ is holomorphic if and only if $\om\overline\om\equiv0$. 
\end{thm}

\begin{proof} Observe first that if $\om\overline\om\equiv0$ then a differentiation (after a conjugation) shows that $\overline\om D\om\equiv0$. In view of \eqref{def-pre-S-form} we have that $P_f=Ph$ which is clearly holomorphic. 

Conversely, if $P_f$ is holomorphic then, by formula \eqref{def-pre-S-form}, we have that
$$
0 \, = \, \overline{D}P_f\langle \cdot,\cdot,\cdot \rangle \, = \, - Dh^{-1} \overline{D}\big[ (I_n - \overline{\om}\om )^{-1} \, \overline{\om} \big]\, \big\langle \cdot, D\om\langle \cdot,Dh \, \cdot \rangle \big\rangle. 
$$
Since $Dh$ is invertible this is equivalent to 
\begin{equation} \label{D-bar-calc-1}
\overline{D}\big[ (I_n - \overline{\om}\om )^{-1} \, \overline{\om} \big]\, \big\langle \cdot, D\om\langle \cdot, \cdot \rangle \big\rangle \, = \, 0. 
\end{equation}
Using the differentiation properties \eqref{prod-rule} and \eqref{deriv-inverse} we compute 
\begin{align*}
\overline{D}\big[ (I_n - \overline{\om}\om )^{-1} \, \overline{\om} \big]\langle \cdot, \cdot \rangle  \, & = \, (I_n - \overline{\om}\om )^{-1} \overline{D\om}\langle \cdot, \om(I_n - \overline{\om}\om )^{-1}\overline{\om} \, \cdot \rangle + (I_n - \overline{\om}\om )^{-1}\overline{D\om}\langle \cdot, \cdot \rangle \\
& = \, (I_n - \overline{\om}\om )^{-1} \overline{D\om}\langle \cdot, \big( \om\overline{\om}(I_n - \om\overline{\om} )^{-1} +I_n \big)  \cdot \rangle \\
& = \, (I_n - \overline{\om}\om )^{-1} \overline{D\om}\langle \cdot,  (I_n - \om\overline{\om} )^{-1} \, \cdot \rangle,  
\end{align*}
so that \eqref{D-bar-calc-1} is equivalent to 
\begin{equation} \label{D-bar-calc-2}
\overline{D\om} \big\langle \cdot, (I_n - \om\overline{\om} )^{-1} D\om\langle \cdot, \cdot \rangle \big\rangle \, = \, 0. 
\end{equation}
In order to differentiate \eqref{D-bar-calc-2} with respect to $z$ we compute
\begin{align*}
D\big[ (I_n - \om\overline{\om} )^{-1} D\om \big]\langle \cdot, \cdot, \cdot \rangle  \,  = & \, (I_n - \om\overline{\om} )^{-1} D\om \big\langle \cdot, \overline{\om} (I_n - \om\overline{\om} )^{-1} D\om \langle \cdot, \cdot \rangle  \big\rangle \\ 
& + (I_n - \om\overline{\om} )^{-1} D^2\om \langle \cdot, \cdot, \cdot \rangle 
\end{align*}
and apply to this the operator $\overline{D\om}$. Then a further usage of \eqref{D-bar-calc-2} shows that the first summand vanishes so that we obtain 
$$
\overline{D\om} \big\langle \cdot, (I_n - \om\overline{\om} )^{-1} D^2\om\langle \cdot, \cdot, \cdot \rangle \big\rangle \, = \, 0. 
$$
This can be repeated an arbitrary number of times, so that after $k-1$ differentiations of \eqref{D-bar-calc-2} we have that 
\begin{equation} \label{D-bar-calc-3}
\overline{D\om} \big\langle \cdot, (I_n - \om\overline{\om} )^{-1} D^k\om\langle \cdot,\ldots, \cdot \rangle \big\rangle \, = \, 0, \qquad k\geq 1, 
\end{equation}
holds throughout the domain $\Om$. According to the Taylor formula centered at some $\al $ in $\Om$ we have that
\begin{equation} \label{Taylor}
\om(z) \,= \, \sum_{k=0}^\infty \frac{1}{k!} D^k\om(\al) (z-\al)^k, \qquad |z-\al|<\de(\al), 
\end{equation}
where $\delta(\al) = {\rm dist}(\al, \partial\Om)$. We apply to this the operator $\overline{D\om} \langle \cdot, (I_n - \om\overline{\om} )^{-1} \cdot \rangle $ evaluated at $\al$ so that, in view of \eqref{D-bar-calc-3}, we obtain 
\begin{align*}
\overline{D\om(\al)} \big\langle \cdot, \big(I_n - \om(\al)\overline{\om(\al)} \,\big)^{-1}\om(z) \cdot \big\rangle & \,= \, \overline{D\om(\al)} \big\langle \cdot, \big(I_n - \om(\al)\overline{\om(\al)} \,\big)^{-1}\om(\al)\cdot \big\rangle \\
& \,= \, \overline{D\om(\al)} \big\langle \cdot, \om(\al) \big(I_n - \overline{\om(\al)}\om(\al) \big)^{-1} \cdot \big\rangle. 
\end{align*}
Since $\al\in\Om$ is arbitrary we may take $z=z_0$ in order to use the assumption $\om(z_0)=0$. Denoting by 
$$
E(z) \, = \, \{ \al \in\Om \, : \, |\al-z| <\delta(\al)\}
$$
the set of points in $\Om$ which lie closer to $z$ than to $\partial \Om$ (that is, $\partial E(z)$ consists of points which are equidistant to $z$ and $\partial \Om$), we get that 
$$
\overline{D\om(\al)} \big\langle \cdot, \om(\al) \big(I_n - \overline{\om(\al)}\om(\al) \big)^{-1} \cdot \big\rangle \,= \, 0, \qquad  \al \in E(z_0).
$$
Equivalently, we have that $D\om \langle \cdot, \overline{\om}\, \cdot \rangle = 0$ in $E(z_0)$. After $k-1$ differentiations we get that 
$$
D^k\om \langle \cdot,\ldots, \cdot, \overline{\om}\, \cdot \rangle = 0, \qquad k\geq 1, 
$$ 
in $E(z_0)$. We consider again the Taylor formula \eqref{Taylor} centered at some $\al \in E(z_0)$ and multiply it on the right with $\overline{\om(\al)}$ in order to obtain 
$$
\om(z)\overline{\om(\al)} \, = \, \om(\al)\overline{\om(\al)}, \qquad |z-\al|<\de(\al).
$$
Taking $z=z_0$ we conclude that $\om\overline{\om}=0$ in $E(z_0)$. In view of the identity principle for real analytic functions (see Corollary 2.3.8 in \cite{Kr}) we have that $\om\overline{\om}=0$ throughout $\Om$. 
\end{proof}

Note that if $f\in\cp(\Om)$ does not satisfy the second assumption of Theorem~\ref{vanish-pre-S} then we may normalize it by means of the affine transformation $F=f-\overline{\om(z_0)}\, \overline{f}$, so that $\om_F(z_0)=0$, while at the same time having $P_F=P_f$ in view of Theorem~\ref{harm-afin-inv}. Since the point $z_0\in\Om$ is arbitrarily chosen it is easy to see that $P_f$ is holomorphic if and only if 
$$
\big(\om(\al) - \om(\beta) \big) \big(I_n - \overline{\om(\beta)}\om(\al)\big)^{-1} \big(\, \overline{\om(\al)} - \overline{\om(\beta)}\, \big) \, = \, 0, \quad \text{for all} \;\; \al,\beta\in\Om. 
$$

As a direct consequence of Theorem~\ref{vanish-pre-S}, the case when $P_f$ vanishes is characterized in the following proposition. 
\begin{cor}  \label{cor-vanish-pre}
Let $f\in\cp(\Om)$ be such that $\om(z_0)=0$ for some $z_0\in \Om$. Then $P_f\equiv 0$ if and only if $\om\overline\om\equiv0$ and $h$ is linear. 
\end{cor}

In order to see if any condition on $P_f$ can imply that $f$ is univalent we recall now Corollary 2.2 from \cite{CHHK14}. We formulate it here for a general simply connected domain $\Om\subset\C^n$ (since its proof does not depend on the domain of definition), even though in \cite{CHHK14} it was stated for the unit ball in $\C^n$.

\begin{thmother}[\cite{CHHK14}] \label{Chuaqui-HHK}
Let $h$ be a biholomorphic mapping in $\Om$ for which $h(\Om)$ is convex and let $f=h+\overline{g}$ be a pluriharmonic mapping for which $\|\om\|<1$ in $\Om$. Then $f$ is univalent. 
\end{thmother}

With the aid of this we can prove the following. 

\begin{cor}
Let $\Om$ be a convex domain  and $f\in\cp(\Om)$ be such that $\|\om\|<1$ in $\Om$ and $\om(z_0)=0$ for some $z_0\in \Om$. Then $P_f\equiv 0$ implies that $f$ is univalent. 
\end{cor}

\begin{proof}
Since $h$ is linear by Corollary \ref{cor-vanish-pre}, its image is a convex domain and Theorem~\ref{Chuaqui-HHK} may be applied. 
\end{proof}

\subsection{Stability for the pre-Schwarzian} Here we turn to the notion of stability as it was introduced in the plane in \cite{HM13} and generalized to the following form in \cite{CHHK14}. A property is said to be \emph{stable} if it is shared by $f=h+\overline{g}$ and $F=h+A\overline{g}$ for all unitary matrices $A$.  Similarly, we say that an expression is stable if it is invariant under the transformation $f\mapsto h+A\overline{g}$. We show that the pre-Schwarzian is stable under rotations of the identity and that these are the only unitary matrices with this property. 

We denote the unit circle by $\T=\{z\in\C:|z|=1\}$. Also, we denote by $e_j$ the vectors of the standard basis and by $E_{ij}$ the matrix whose only non-zero entry is at the $(i,j)$-position and is equal to 1. We indicate with $R_k$ on the left and $C_k$ on the top of a matrix the position of the k-th row and column, respectively. For example, the matrix $E_{ij}$ could be given by
$$
E_{ij} \, = \,  \kbordermatrix{
         &   & C_j       \\
    R_i &  &   1   \\
       &  &     },  
$$
where all entries that do not appear are equal to zero. 

In the following the standard notation $f=h+\overline{g}$ and $F=h+A\overline{g}$ is used. 

\begin{thm}  \label{stable-thm}
Let $A$ be a unitary matrix. Then $P_F=P_f$ for all $f\in\cp(\Om)$ with $\|\om\|<1$ if and only if $A=\la I_n$ for some $\la\in\T$. 
\end{thm}

\begin{proof} 
The dilatation of $F$ is given by $\om_F=\overline{A} \om$. Hence, in view of definition \eqref{def-pre-S-form}, it is evident that $P_F=P_f$ if $A$ is a rotation of the identity.

For the converse we give counterexamples in the polydisk $\Om=\D^n$. Assume first that the matrix $A=(a_{ij})_{ij}$ has at least one non-zero entry off the diagonal, that is, $a_{ij}\neq0$ for some $i\neq j$. We consider the mapping $f=h+\overline{g}$ for which $h(z)=z$ and $g(z)=\frac{1}{2}z_i^2 e_j$. We see that $\om = z_i E_{ji}$ satisfies $\|\om\| = |z_i|<1$ and, moreover, that $\om\overline{\om}=0$, so that $P_f=0$. The dilatation of the transformed mapping $F$ is given by $\om_F \, = \, \overline{A} \om \, = \, z_i  M$, where $M$ is the matrix with non-zero elements only in the $i$-th column $C_i= (\overline{a_{1j}}, \ldots, \overline{a_{nj}})$. Since $\om_F(0)=0$ and $\om_F\overline{\om_F}=a_{ij}\overline{z_i}\om_F$ is not zero, Corollary \ref{cor-vanish-pre} can be applied to deduce that $P_F$ is not zero and, therefore, distinct from $P_f$. 

For the remaining case we may assume that $A$ is a diagonal matrix with entries $\la_k\in\T, k=1,2,\ldots,n$, and such that $\la_i\neq\la_j$ for some $i<j$. We first prove that, in general, $P_F=P_f$ is equivalent to 
\begin{equation} \label{stable-equiv-formula}
( A \overline{\om} \overline{A} -\overline{\om}) (I_n- \om\overline{\om})^{-1} D\om \, = \, 0.
\end{equation}
Indeed, in view of \eqref{def-pre-S-form}, $P_F=P_f$ is equivalent to 
$$
\big(I_n - A\overline{\om}\overline{A}\om \big)^{-1} \,A \overline{\om} \overline{A} \, D\om  \, =\, \big(I_n - \overline{\om}\om \big)^{-1} \, \overline{\om} \, D\om, 
$$
which is the same as
\begin{align*}
0 \, = \, & \big[ A \overline{\om} \overline{A} -\big(I_n - A\overline{\om}\overline{A}\om \big) \big(I_n - \overline{\om}\om \big)^{-1} \, \overline{\om}\, \big]  D\om \\
= \, & \big[ A \overline{\om} \overline{A} -\big(I_n - A\overline{\om}\overline{A}\om \big) \overline{\om} \big(I_n - \om \overline{\om} \big)^{-1} \big]  D\om \\
= \, & \big[ A \overline{\om} \overline{A} \big(I_n - \om \overline{\om} \big) -\big(I_n - A\overline{\om}\overline{A}\om \big) \overline{\om} \,  \big] \big(I_n - \om \overline{\om} \big)^{-1} D\om \\
= \, & ( A \overline{\om} \overline{A} -\overline{\om}) (I_n- \om\overline{\om})^{-1} D\om,
\end{align*}
which proves our claim \eqref{stable-equiv-formula}. We consider the mapping $f=h+\overline{g}$ with $h(z)=z$ and $g(z)=\frac{1}{2} (z_i^2 e_j + z_j^2 e_i)$. We compute 
$$
\om \, = \,  \kbordermatrix{
         & C_i &   & C_j    \\
    R_i &  &  &  z_j  \\
      &  &  &   \\
    R_j &  z_i &  &        }  \qquad \text{and} \qquad D\om \langle u , v\rangle \, = \,  \kbordermatrix{
         & C_i &   & C_j    \\
    R_i &  &  &  u_j  \\
      &  &  &   \\
    R_j &  u_i &  &       } v. 
$$
Moreover, we find that $\om\overline{\om}=\overline{z_i}z_j E_{ii} + z_i\overline{z_j}E_{jj}$, so that $I_n-\om\overline{\om} $ is a diagonal matrix which we can easily invert to obtain 
$$
(I_n-\om\overline{\om})^{-1}   \, = \,  \kbordermatrix{
       & &  & & C_i &  && & C_j  & && \\
       &1&  & &      &  &&  &      &  &  & \\
       & &\ddots & &   &&   &    &      &  &  & \\
       & & & 1 &      &  &&  &      &  &  & \\
  R_i & & &&  \frac{1}{1-\overline{z_i}z_j} &  &&    &  &&& \\
       & &  & &     &1&       &    &&& &\\
       & &  & &     &  &  \ddots    &   &&& &\\     
       & &  & &     &  &                & 1  &&& &\\
  R_j & & &&    &&& &  \frac{1}{1-z_i\overline{z_j}}     && &  \\  
        & & &  &    &&&  &    &      1 &     & \\
       &  &&  &    &  & &&   &     &  \ddots &   \\
        & & &  &  &    &  && &         &      &1 }.    
$$
Finally, we have 
$$
A \overline{\om} \overline{A} -\overline{\om} \, = \,   \kbordermatrix{
         & C_i &   & C_j    \\
    R_i &  &  &  (\la_i\overline{\la_j}-1)\overline{z_j}  \\
      &  &  &   \\
    R_j &  (\overline{\la_i}\la_j-1)\overline{z_i}  &  &        }, 
$$
with which we conclude that 
$$
( A \overline{\om} \overline{A} -\overline{\om}) (I_n- \om\overline{\om})^{-1} D\om\langle u , v\rangle  \, =  \,   \kbordermatrix{
         & C_i &   & C_j    \\
    R_i &   \frac{(\la_i\overline{\la_j}-1)u_i\overline{z_j}}{1-z_i\overline{z_j}}  &  &  \\
      &  &  &   \\
    R_j &   &  &  \frac{(\overline{\la_i}\la_j-1) u_j\overline{z_i}  }{1-\overline{z_i}z_j}       }, 
$$
Since this is not zero we get that $P_f$ and $P_F$ are distinct in view of \eqref{stable-equiv-formula}. 
\end{proof}

\vskip.5cm
\section{Schwarzian derivative} \label{sect-3}

\subsection{The Schwarzian derivative for holomorphic mappings in $\C^n$} For a locally biholomorphic mapping $f$ in a domain $\Om\subset\C^n$, Oda \cite{Od74} defined the Schwarzian derivatives 
\begin{equation} \label{Oda}
S_{ij}^kf \, = \, \sum_{\ell=1}^n \frac{\partial^2 f_\ell }{\partial z_i \partial z_j} \frac{\partial z_k}{\partial f_\ell} - \frac{1}{n+1} \left( \de_{ik}  \frac{\partial}{\partial z_j} +\de_{jk}  \frac{\partial}{\partial z_i} \right) \log J_f, \qquad 
\end{equation}
for $i,j,k=1,\ldots,n$; here $\de_{ij}$ is Kronecker's delta. These differential operators satisfy a certain chain rule formula and, also, in dimension $n\geq2$ they all vanish only for M\"obius transformations, \textit{i.e.}, for the mappings 
\begin{equation} \label{Moebius-n}
T(z) \,= \, \left( \frac{\ell_1(z)}{\ell_0(z)}, \ldots,  \frac{\ell_n(z)}{\ell_0(z)} \right), \qquad z\in\C^n, 
\end{equation}
where $\ell_i(z)=a_{i0} + a_{i1}z_1+\ldots + a_{in}z_n, \, i=0,\ldots,n$, with ${\rm det}(a_{ij})_{ij}\neq 0$. Note that, in contrast to the planar case, these are differential operators of order 2. This is explained by the fact that in the plane the order of the M\"obius group is 3 and, therefore, any M\"obius-invariant operator should have order at least 3. However, in dimension $n\geq2$ the parameters involved in the value, the first and the second derivatives of $f$ are $n^2(n+1)/2+n^2+n$, already exceeding the order of the M\"obius group which is $n^2+2n$. The difference of the two orders is $n(n-1)(n+2)/2$ and this is precisely the number of independent terms $S_{ij}^kf$, which can be counted considering that they satisfy  $S_{ij}^kf  = S_{ji}^kf$ for all $k$ and $\sum_{j=1}^nS_{ij}^j f =0$. 

\vskip.15cm

With the matrices $\SS^kf= \big( S_{ij}^kf \big)_{ij}$, for $k=1,\ldots,n$, the third author \cite{He06} defined the symmetric bilinear operator
\begin{equation} \label{def schw op}
Sf(z)\langle u,v\rangle \,  = \, \big( u^t\SS^1f(z) v, \ldots,  u^t\SS^n f(z) v\big), \qquad z\in\Om, \; u,v\in\C^n;  
\end{equation}
here $u^t$ denotes the transpose of the vector $u$. A straightforward calculation then produces the expression 
\begin{equation}  \label{def-Schw-n-holom}
Sf(z)\langle u,v\rangle = Pf(z)\langle u,v\rangle - \frac{1}{n+1}   \Big( \big( \nabla\log J_f(z) \cdot u \big)v  +\big( \nabla\log J_f(z) \cdot v \big)u \Big), 
\end{equation}
where $Pf$ is the pre-Schwarzian derivate \eqref{Pfhol} and $\nabla=(\partial/\partial z_1, \ldots , \partial/\partial z_n)$ is the complex gradient operator. A way to verify that \eqref{def schw op} and \eqref{def-Schw-n-holom} are equivalent is to consider the vectors $u=e_i$ and $v=e_j$ of the standard basis and see that in this case both expressions are equal to $\big(S_{ij}^1f(z),\ldots ,S_{ij}^nf(z)\big)$.

\subsection{Definition of $S_f$ for pluriharmonic mappings} For $f=h+\overline{g}$ in $\cp(\Om)$, the class of pluriharmonic mappings given in \eqref{class-loc-univ-pluri}, we define the Schwarzian derivative $S_f$ in complete analogy to \eqref{def-Schw-n-holom}, employing the pre-Schwarzian derivate \eqref{def-pre-S-form} and the Jacobian \eqref{Jacob-pluri} of $f$. We now find two equivalent formulations of this definition. With a brief calculation (and suppressing the variable $z\in\Om$) we get that 	
\begin{align}  \label{def-Schw-n-pluri}
S_f\langle u,v\rangle \, =  \, Sh \langle & u,v \rangle - Dh^{-1} (I_n - \overline{\om}\om )^{-1} \, \overline{\om} \, D\om\langle u,Dh\, v \rangle  \\ 
-\frac{1}{n+1} &\bigg( \Big\{ \nabla\log \big[{\rm det} (I_n - \om\overline{\om} ) \big] \cdot u \Big\}v  +\Big\{ \nabla\log \big[{\rm det} (I_n - \om\overline{\om} ) \big] \cdot v \Big\}u \bigg).  \nonumber 
\end{align}
Also, using Jacobi's formula for the derivative of a determinant (see \cite{Go72}), we get from \eqref{Jacob-pluri} that 
$$
\frac{\partial}{\partial z_k} \left( \log J_f \right) \, =  \,  {\rm Tr}\left( Dh^{-1} \frac{\partial (Dh)}{\partial z_k} \right) - {\rm Tr}\left( (I_n - \om\overline{\om} )^{-1} \frac{\partial \om}{\partial z_k} \,\overline{\om}  \right), 
$$
where  ${\rm Tr}(\cdot)$ denotes trace. Observe that the linear mappings $D^2h\langle e_k,\cdot \rangle$ and $D\om\langle e_k,\cdot \rangle$ can be interpreted as the matrices $\partial (Dh)/\partial z_k$ and $\partial \om / \partial z_k$, respectively. Using standard properties of the trace we get 
\begin{align*} 
\frac{\partial}{\partial z_k} \left( \log J_f \right) \, = & \,  {\rm Tr}\left( Dh^{-1} D^2h\langle e_k,\cdot \rangle \right) - {\rm Tr}\left( \overline{\om}  (I_n - \om\overline{\om} )^{-1} D\om\langle e_k,\cdot \rangle\,\right) \\
= & \,  {\rm Tr}\left( Ph \langle e_k,\cdot \rangle \right) - {\rm Tr}\left( Dh^{-1}  (I_n - \overline{\om}\om )^{-1} \overline{\om} D\om\langle e_k,Dh \cdot \rangle\,\right) \\
= & \,  {\rm Tr}\left( P_f \langle e_k,\cdot \rangle \right).
\end{align*}
Therefore, once again from \eqref{def-Schw-n-holom}, we obtain
\begin{align} \label{def-Schw-n-pluri-alt}
 S_f\langle u,v\rangle & \,  = \, P_f\langle u,v\rangle \\ 
& -\frac{1}{n+1} \Bigg( \bigg\{ \Big[ \sum_{k=1}^n {\rm Tr} \big(P_f\langle e_k , \cdot \rangle\big) e_k \Big] \cdot u \bigg\}v  +\bigg\{ \Big[ \sum_{k=1}^n {\rm Tr} \big(P_f\langle e_k , \cdot \rangle\big) e_k \Big] \cdot v \bigg\}u  \Bigg),  \nonumber 
\end{align}
which shows that $S_f$ can be written only in terms of $P_f$. 

\subsection{Basic theorems}  In view of formula \eqref{def-Schw-n-pluri-alt} the following theorem is a direct consequence of the corresponding theorems on the pre-Schwarzian. 

\begin{thm} Let $f=h+\overline g$ be a pluriharmonic mapping in $\Om$ having dilatation $\omega$. Then 

\vskip.2cm
\begin{enumerate}
\item[(i)]$S_{f+A\overline f}=S_f$ for every matrix $A$ with $\|A\|<1$, \\

\item[(ii)]$S_{Bf}=S_f$ for every invertible matrix $B$, \\

\item[(iii)]$S_{h+\lambda \overline g}=S_f$ for every complex number $\lambda$ with $|\lambda|=1$, \\

\item[(iv)]$S_f(z_0)=S(h-\overline{\omega(z_0)}g)(z_0)$ for any fixed $z_0\in\Omega$.
\end{enumerate}
\end{thm}

\vskip.15cm
\begin{proof} 
Statements $(i), (ii), (iii)$ and $(iv)$ follow directly from Theorem~\ref{harm-afin-inv}, Theorem~\ref{harm-mult-inv}, Theorem~\ref{stable-thm} and Lemma~\ref{pre-S-lem}, respectively. 
\end{proof}

\begin{thm}[Chain rule for $S_f$]
Let $f=h+\overline{g}$ be a pluriharmonic mapping and $\varphi$ be a locally biholomorphic mapping such that the composition $f\circ\varphi$ is well defined. Then 
$$
S_{f\circ\varphi}(z)\langle \cdot,\cdot \rangle \, = \, D\varphi(z)^{-1} S_f\big(\varphi(z)\big) \langle D\varphi(z)\, \cdot, D\varphi(z) \, \cdot \rangle +S\varphi(z) \langle\cdot,\cdot\rangle 
$$
for all $z$ in the domain of definition of $\varphi$. 
\end{thm}
\begin{proof}
In view of the chain rule for the pre-Schwarzian given in Theorem~\ref{chain-rule-pre-Schw} and the formula 
$$
J_{f\circ\phi}(z) \, = \, J_f\big(\phi(z)\big) \, J_\phi(z)
$$
satisfied by the Jacobian of a composition, our definition \eqref{def-Schw-n-holom} readily yields 
\begin{align*} 
S_{f\circ\phi}(z)\langle u,v\rangle & = S\phi(z)\langle u,v\rangle +  D\varphi(z)^{-1} P_f\big(\varphi(z)\big) \langle D\varphi(z) u, D\varphi(z) v \rangle \\
& - \frac{1}{n+1}   \Big( \big( \nabla(\log J_f\circ\phi)(z) \cdot u \big)v  +\big( \nabla(\log J_f\circ\phi)(z) \cdot v \big)u \Big).
\end{align*}
To get the desired result it suffices to verify that 
$$
\nabla(\log J_f\circ\phi)(z) \cdot u \, = \, \nabla(\log J_f)\big(\phi(z)\big) \cdot \big(D\phi(z)u\big)
$$
for the vectors $u=e_k$ of the standard basis or, equivalently, that
$$
\frac{\partial\left( J_f \circ\phi \right)}{\partial z_k} (z)  \, = \, \sum_{j=1}^n \frac{\partial J_f}{\partial z_j} \left( \phi(z) \right)  \frac{\partial \phi_j}{\partial z_k}(z),
$$
which is evident. 
\end{proof}

It is easy to see in \eqref{def-Schw-n-pluri} that if the dilatation $\om$ is constant then $S_f=Sh$ and, furthermore, that if $f=T + A \, \overline{T}$ for a M\"obius transformation $T$ as in \eqref{Moebius-n} and a matrix $A$ with $\|A\|<1$ then $S_f=ST=0$. In the converse direction we have the following theorem. 

\begin{thm} \label{analytic-Schw}
Let $f=h+\overline g\in \cp(\Om)$ be such that $\omega(z_0)=0$ for some $z_0\in\Omega$, and let also $\vp$ be a holomorphic mapping in $\Omega$. Then $S_f =\vp $ implies that $Sh =\vp$.
\end{thm}

\begin{proof} Observe that 
$$
\frac{\partial}{\partial z_i}\log \big[{\rm det}\big(I_n - \om\overline{\om} \big) \big] =\frac{1}{{\rm det}\big(I_n - \om\overline{\om} \big)}\sum_{k=1}^n{\rm det}\big(F_1,\ldots, \frac{\partial F_k}{\partial z_i},\ldots,F_n\big),
$$ 
where $F_k$'s are the rows of $I_n - \om\overline{\om}$. But $F_k=e_k-\om_k \overline\om$, where $\om_k$ is the $k$-row of $\om$, thus 
$$
\frac{\partial}{\partial z_i}\log \big[{\rm det}\big(I_n - \om\overline{\om} \big) \big] =\frac{1}{{\rm det}\big(I_n - \om\overline{\om} \big)}\sum_{k=1}^n{\rm det}\big(e_1-\om_1 \overline\om,\ldots, -\frac{\partial \om_k}{\partial z_i}\overline\om,\ldots,e_n-\om_n\overline\omega\big).
$$ 
Hence this expression vanishes at $z_0$ since $\omega(z_0)=0$. Moreover, differentiation with respect to $z_j, 1 \leq j \leq n$, of each row in the determinant in the summation above will leave the term $\overline{\om}$ intact, that is, 
$$
\frac{\partial}{\partial z_j} \left( - \frac{\partial \om_k}{\partial z_i}\overline\om   \right) =  - \frac{\partial^2 \om_k}{\partial z_j \partial z_i }\overline\om \qquad \text{and} \qquad \frac{\partial}{\partial z_j} (e_\ell-\om_\ell \overline\om) =- \frac{\partial \om_\ell}{\partial z_j}\overline\om, \quad \ell\neq k. 
$$
Therefore, we have that  
\begin{equation} \label{nabla det log}
\frac{\partial^m {\rm det} \big(e_1-\om_1 \overline\om,\ldots, -\frac{\partial \om_k}{\partial z_i}\overline\om,\ldots,e_n-\om_n\overline\omega\big) }{ \partial z_1^{m_1}\cdots \partial z_n^{m_n} }  (z_0) =0, \quad m\geq0, 
\end{equation}  
where $m=m_1+\ldots+m_n$, again because $\omega(z_0)=0$. 

On the other hand, we can rewrite $D^n\big(Dh^{-1}(I_n-\overline\om\om)^{-1}\overline\om D\om Dh\big)$ as 
$$
\sum_{k=0}^n {n \choose k} D^{n-k}\big(Dh^{-1}(I_n-\overline\om\om)^{-1}\big)\, \overline\om \, D^k( D\om Dh).
$$ 
Hence, 
\begin{equation}\label{Derivada n}  
D^n\big(Dh^{-1}(I_n-\overline\om\om)^{-1}\overline\om D\om Dh\big)(z_0) =0. 
\end{equation} 
Now, using the equations (\ref{def-Schw-n-pluri}), (\ref{nabla det log}), and (\ref{Derivada n}) we have that  $D^n\vp(z_0)=D^nS_f(z_0)=D^nSh(z_0)$ for all $n\geq 0$ and, therefore, $Sh=\vp$.
\end{proof}

As a direct consequence of Theorem~\ref{analytic-Schw} for $\vp\equiv0$, and in view of the fact that the Schwarzian derivative \eqref{def schw op} vanishes only for M\"obius transformations \eqref{Moebius-n}, we get the following proposition. 

\begin{cor}
Let $f=h+\overline g\in \cp(\Om)$ be such that $\omega(z_0)=0$ for some $z_0\in\Omega$. Then $S_f \equiv 0$ implies that $h$ is a M\"obius transformation.
\end{cor}

\section{Miscellaneous} \label{sect-4}
In this section we give examples which reveal some interesting differences between harmonic mappings in the plane and pluriharmonic mappings in higher complex dimensions.

\subsection{Dilatation of sense-preserving pluriharmonic mappings} According to Theorem~\ref{DHK-om-sense-preserv}, a pluriharmonic mapping $f=h+\overline{g}$ is sense-preserving at a point $z$ if its dilatation satisfies $\|\om(z)\|<1$. It is natural to ask if, conversely, the condition $J_f>0$ can imply some bound on the norm of the dilatation. We now see that this is not possible. 

\begin{example} Let $h$ be a locally biholomorphic mapping in a domain in $\C^2$ and consider the constant dilatation
\begin{equation*}
\om \, = \, \left(
\begin{array}{cc}
0  & t \\
-1  & 0
\end{array} \right), \qquad t\geq1.   
\end{equation*}
We readily compute $I_2 - \om \overline{\om} = (1+t)I_2$, hence $J_f = |{\rm det}Dh|^2 (1+t)^2 >0$. But $\|\om\| = t$, which can be arbitrarily large. The pluriharmonic mapping $f=h+\overline{g}$ for which $g = (g_1,g_2) = (t h_2, -h_1)$ satisfies the above since, indeed, $Dg=\omega Dh$. 

It is easy to generalize this example to any dimension $n\geq2$. Simply take $\om$ to be the $n\times n$ matrix with only two non-zero entries: $t$ in the upper-right corner and $-1$ in the lower-left corner. 
\end{example}

\subsection{Dilatation of affine transformations} Let $f=h+\overline{g}$ be a pluriharmonic mapping and let $A$ be a matrix with $\|A\|<1$. Then, in view of \eqref{om_F}, the affine transformation $F = f+A \,\overline{f}$ has dilatation $\om_F \, =  \, (\om + \overline{A})(I_n+A\,\om )^{-1}$. We now give a counterexample proving that $\|\om\|<1$ does not imply $\|\om_F\|<1$. Note that the choice of $A$ in this example is arguably the most useful since it ensures that the mapping $F$ satisfies the normalization $\om_F(0)=0$.

\begin{example} \label{counter-omega}
Let $\vp$ be an automorphism of the unit disk $\D=\{z\in\C:|z|<1\}$ given by $\vp(z)=(\al+z)/(1+\al z)$, for $\al\in(0,1)$, and consider the matrix valued mapping 
\begin{equation*}
\om(z,w) \, = \, \frac{\vp(z) }{\sqrt{2}} \left(
\begin{array}{cc}
1 & 1 \\
 0 &  0
\end{array} \right), \qquad  (z,w)\in\D\times\Om_2, 
\end{equation*}
for some domain $\Om_2\subset\C$ containing the origin. We readily compute that $\|\om(z,w)\| = |\vp(z)| <1$. Consider also the matrix 
\begin{equation*}
A \, = \,  - \overline{\om(0,0)} \, = \, - \frac{\al }{\sqrt{2}} \left(
\begin{array}{cc}
1 & 1 \\
 0 &  0
\end{array} \right)
\end{equation*}
and the affine transformation $F = f+A \,\overline{f}$. We compute 
\begin{equation*}
\om + \overline{A} \, = \, \frac{\vp - \al}{\sqrt{2}} \left(
\begin{array}{cc}
1 & 1 \\
 0 &  0
\end{array} \right) 
\end{equation*}
and 
\begin{equation*}	
I_2+A\,\om  \, = \,  \frac{1}{2}\left(
\begin{array}{cc}
2 - \al \vp & -  \al \vp  \\
 0 &  2
\end{array} \right), 
\end{equation*}
whose inverse is given by 
\begin{equation*}	
( I_2+A\,\om )^{-1} \, = \, \frac{1}{ 2- \al \vp} \left(
\begin{array}{cc}
2 & \al \vp \\
 0 &  2 - \al \vp 
\end{array} \right).
\end{equation*}
Therefore, we have that 
\begin{equation*}	
\om_F \, = \, \frac{\sqrt{2} \,(\vp - \al) }{ 2- \al \vp} \left(
\begin{array}{cc}
1 & 1 \\
 0 &  0
\end{array} \right), 
\end{equation*}
whose norm is
$$
\|\om_F(z,w)  \| \, = \,  \frac{2 \, |\vp(z) - \al | }{ |2- \al \vp(z) | } \, = \, \frac{2(1-\al^2)|z|}{| 2-\al^2+\al z | }. 
$$
We compute
$$
\|\om_F\|_\infty \, = \, \sup \{ \|\om_F(z,w)  \|  \, : \, (z,w) \in \D\times\Om_2 \}  \, = \, \frac{2(\al+1)}{\al+2} 
$$
and see that this increases with $\al\in(0,1)$ from 1 to $ 4/3$. 

An obvious modification of this example to any dimension $n\geq2$ would be to consider the dilatation 
$$
\om(z) \, = \, \frac{\vp(z_1) }{\sqrt{n}} \, B, \qquad z_1\in\D, 
$$
with the matrix $B$ having all the entries of its first row equal to 1 and all the remaining entries equal to 0. Then, with the same choice of $\vp$ and $A$, it is not difficult to compute that $\|\om_F\|_\infty=n(\al+1)/(\al+n)$, which increases with $\al\in(0,1)$ from 1 to $ 2n/(n+1)$. 

\end{example}

The following example shows that under only the assumption ${\rm det}( I_n - \om \overline{\om}) > 0$ it is possible that the mapping $\om_F$ is not even well defined. 

\begin{example} \label{counter-det}
Let $t\in(0,1)$ and consider the matrices 
\begin{equation*}	
A \, = \,  \left(
\begin{array}{cc}
t & 0 \\
 0 & -t
\end{array} \right) \qquad \text{and} \qquad \om \, = \,  \left(
\begin{array}{cc}
0 & 1/t^2 \\
 -1 & 0
\end{array} \right), 
\end{equation*}
interpreting $\om$ as the dilatation of the pluriharmonic mapping 
$$
f(z,w) \, =\, \left(z+\frac{\overline{w}}{t^2} \,, \, w-\overline{z} \right),
$$
for example. We easily compute that $\|A\|=t$ and 
\begin{equation*}	
I_n -\om \overline{\om} \, = \,  \left(
\begin{array}{cc}
1+1/t^2 & 0 \\
0 & 1+1/t^2
\end{array} \right), 
\end{equation*}
so that ${\rm det}(I_n -\om \overline{\om} ) = (1+1/t^2)^2 >0$. Since the matrix 
\begin{equation*}	
I_2 +A\om  \, = \,  \left(
\begin{array}{cc}
1 & 1/t \\
t & 1
\end{array} \right)
\end{equation*}
is singular we deduce that the holomorphic part $H$ of the affine transformation $F = f+A \,\overline{f}$ is everywhere singular, that is, it satisfies ${\rm det}DH\equiv 0$, and therefore it is not possible to define the dilatation of $F$.
\end{example}

We note that the proof of Theorem~\ref{factor} shows that the condition ${\rm det}( I_n - \om \overline{\om}) > 0$ together with the condition ${\rm det}(I_n+A\,\om)\neq0$ for every matrix $A$ with $\|A\|<1$ are sufficient for ${\rm det}( I_n - \om_F \overline{\om_F}) > 0$. However, this does not give an answer to Problem~\ref{prob-1} since the latter of these conditions is not preserved under affine transformations. 

\subsection{No shear construction in $\C^n$} A simply connected domain in the plane is called convex in the horizontal direction (CHD) if its intersection with any horizontal line is connected or empty. Let $f=h+\overline{g}$ be a locally univalent planar harmonic mapping (the domain of definition is not relevant here). Then according to Clunie and Sheil-Small's ``shear construction'' introduced in  \cite{CS} (see also \cite[\S3.4]{Du3}), $f$ is univalent and its range is CHD if and only if $h-g$ has the same properties. The key to proving this theorem is the following lemma. 

\begin{lemother}[\cite{CS}]\label{shear-lem}
Let a domain $\Om$ be CHD and let $p$ be a real-valued continuous function on $\Om$. Then the mapping $w\mapsto w+p(w)$ is univalent in $\Om$ if and only if it is locally univalent. If it is univalent then its range is CHD.
\end{lemother}

To generalize the shear construction in several complex variables we would have first to give a suitable analogue of the concept of directional convexity for a domain $\Om$ in $\C^n$ and then to generalize Lemma~\ref{shear-lem} to any continuous $p:\Om \to\R^n$. However, we now give an example of a convex domain in $\C^2$ and a function $p$ for which $w\mapsto w+p(w)$ is locally univalent but not univalent, thus excluding the possibility of such a generalization. 

\begin{example}
Consider the domain $\Om=\Om_1\times\Om_2 \subset\C^2$, where 
$$
\Om_1 \, = \,  \{w\in\C \, : \, -1 < {\rm Re}\,w  < 1\} 
$$
and
$$
\Om_2 \, = \,  \{w\in\C \, : \, -(\pi+\ve) < {\rm Re}\,w < \pi+\ve\}, 
$$
for some $\ve>0$. With the notation $w=(w_1,w_2)$ and $w_k=x_k+iy_k, \, k=1,2$, let
$$
p(w) \, = \,  e^{x_1} (\cos x_2 , \sin x_2) - (x_1 , x_2). 
$$
Setting $q(w)=w+p(w)$ we may interpret $q$ as a mapping in $\R^4$ by writing 
$$
q(x_1,y_1,x_2,y_2)  \, = \, (e^{x_1} \cos x_2 , y_1 , e^{x_1} \sin x_2 , y_2),
$$ 
and compute its Jacobian as $J_q(w)=e^{2x_1}>0$. Therefore the mapping $q$ is locally univalent. But, we note that
$$
q(0, -\pi) \, = \, (-1,0) \, = \, q(0,\pi), 
$$
which shows that $q$ is not univalent in $\Om$.

A simple modification of the function $p$, for example, by setting zero in the remaining entries, serves as a counterexample in any dimension $n\geq2$. 
\end{example}

\vskip.3cm
\emph{Acknowledgements}. We wish to thank professor Martin Chuaqui for many fruitful discussions and, in particular, for an observation which significantly reduced the length of the proof of Theorem 5.



\end{document}